\documentclass[12pt]{article}
 \usepackage{a4wide}
 \usepackage{amsthm,amsfonts,amsbsy,graphicx,calrsfs,amssymb}
 \usepackage{amsmath}
 \usepackage{color}

 \theoremstyle{plain}
 \newtheorem{proposition}{Proposition}
 \theoremstyle{definition}
 \newtheorem{example}{Example}

 \DeclareMathOperator\var{\mathrm{var}}
 \DeclareMathOperator\cov{\mathrm{cov}}

 \renewcommand{\le}{\leqslant}
 \renewcommand{\ge}{\geqslant}

 \newcommand{\C}{{\mathbb{C}}}
 \newcommand{\CC}{{\hat {\mathbb{C}}}}
 \newcommand{\D}{\mathcal{D}}
 \newcommand{\T}{\mathcal{T}}
 \newcommand{\TT}{\mathcal{C}}
 \newcommand{\dd}{\mathrm{d}}
 \newcommand{\1}{\mathbf{1}}
 \newcommand{\bu}{\boldsymbol{u}}
 \newcommand{\bv}{\boldsymbol{v}}

\begin{document}
 \title{On the covariance of the \\ asymptotic empirical copula process}
 \author{Christian Genest\footnote{D\'e\-par\-te\-ment de ma\-th\'e\-ma\-ti\-ques et de sta\-tis\-ti\-que, Uni\-ver\-si\-t\'e La\-val, 1045, ave\-nue de la M\'e\-de\-ci\-ne, Qu\'ebec, Canada G1V 0A6. E-mail: Christian.Genest@mat.ulaval.ca}\quad and\quad Johan Segers\footnote{Ins\-ti\-tut de sta\-tis\-ti\-que, bio\-sta\-tis\-ti\-que et scien\-ces ac\-tua\-riel\-les (ISBA), Uni\-ver\-si\-t\'e ca\-tho\-li\-que de Lou\-vain, Voie du Ro\-man Pays 20, B--1348 Louvain-la-Neuve, Belgium. E-mail: Johan.Segers@uclouvain.be} \\[1ex]
 {\it Uni\-ver\-si\-t\'e La\-val and Uni\-ver\-si\-t\'e ca\-tho\-li\-que de Lou\-vain}}
 \maketitle

 \thispagestyle{empty}

\begin{abstract}
Conditions are given under which the empirical copula process associated with a random sample from a bivariate continuous distribution has a smaller asymptotic covariance function than the standard empirical process based on observations from the copula. Illustrations are provided and consequences for inference are outlined.
\end{abstract}


 \section{Introduction}
 \label{section:1}

Consider a pair $(X,Y)$ of continuous random variables whose joint and marginal cumulative distribution functions are defined for all $x, y \in \mathbb{R}$ by
 $$
H(x,y) = \Pr (X \le x, Y \le y), \quad F(x) = \Pr (X \le x),\quad G(y) = \Pr (Y \le y),
 $$
respectively. The transformed variables $U = F(X)$ and $V = G(Y)$ are then uniform on $[0,1]$ and their joint distribution function, defined at every $u, v \in [0,1]$ by
 $$
C(u,v) = \Pr (U \le u, V \le v),
 $$
is the unique copula $C$ associated with $H$. The two functions are related through the equation $ C(u,v) = H\{ F^{-1} (u), G^{-1}(v) \}$, where $F^{-1} (u) = \inf \{ x \in \mathbb{R}: F(x) \ge u \}$ and $G^{-1} (v) = \inf \{ y \in \mathbb{R}: G(y) \ge v \}$ for all $u, v \in (0,1)$. Inference on $C$ is of interest, as it characterizes the dependence in the pair $(X, Y)$; see, e.g., \cite{Joe:1997,Nelsen:2006}. In particular, all margin-free concepts and measures of association such as Kendall's tau, Spearman's rho, Blomqvist's beta, Gini's gamma or Spearman's footrule depend only on $C$.

Let $(X_1, Y_1), \ldots , (X_n, Y_n)$ be a random sample from $H$ and write $U_i = F(X_i)$, $V_i = G(Y_i)$ for all $i \in \{ 1, \ldots , n \}$. When $F$ and $G$ are known, a natural estimate of $C$ is then given by the empirical distribution function of the sample $(U_1, V_1), \ldots , (U_n, V_n)$, defined at every $u, v \in [0,1]$ by
 $$
C_n (u,v) = \frac{1}{n} \sum_{i=1}^n \1 (U_i \le u , V_i \le v).
 $$

In fact, standard results from the theory of empirical processes \cite{Shorack-Wellner:1986} imply that $\C  _n = n^{1/2} (C_n - C)$ converges weakly, as $n \to \infty$, to a centered Gaussian process $\C $ on $[0,1]^2$ with continuous trajectories and covariance function given by
 $$
\cov \{ \C  (u,v), \C  (s,t) \} = C ( u \wedge s, v \wedge t) - C(u,v)C(s,t),
 $$
for all $u, v, s, t \in [0,1]$, with $a \wedge b = \min(a, b)$ for arbitrary $a, b \in \mathbb{R}$.

When the margins $F$ and $G$ are unknown, as is generally the case in practice, they can be estimated by their empirical counterparts, $F_n$ and $G_n$. A surrogate sample from $C$ is then given by the pairs $({\hat U}_1, {\hat V}_1), \ldots , ({\hat U}_n, {\hat V}_n)$, where ${\hat U}_i = F_n (X_i)$ and ${\hat V}_i = G_n (Y_i)$ for all $i \in \{1, \ldots , n \}$. The corresponding empirical distribution function, defined at every $u, v \in [0,1]$ by
 $$
{\hat C} _n (u,v) = \frac{1}{n} \sum_{i=1}^n \1 ({\hat U}_i \le u , {\hat V} _i \le v),
 $$
is traditionally called the empirical copula \cite{Deheuvels:1979}, although it is not a copula \emph{stricto sensu}. The function $\hat C_n$  provides a rank-based, consistent estimate of $C$ often used in practice for copula model selection and goodness-of-fit purposes; see, e.g., \cite{Berg:2009,Genest-Remillard-Beaudoin:2009}.

To be specific, let ${\dot C}_1 (u,v) = \partial C(u,v)/ \partial u$ and ${\dot C}_2 (u,v) = \partial C(u,v) / \partial v$ denote the partial derivatives of an arbitrary copula $C$, known to exist almost everywhere \cite[Chapter 2]{Nelsen:2006}. Now assume that they exist in fact everywhere and that they are continuous on $(0,1)^2$. Under these mild regularity conditions, it is then well known \cite{Rueschendorf:1976, Fermanian-Radulovic-Wegkamp:2004} that the empirical copula process $\mathbb{\hat C}_n = n^{1/2} ({\hat C}_n - C)$ converges weakly, as $n \to \infty$, to a centered Gaussian process $\CC  $ defined at every $u, v \in [0,1]$ by
 $$
\CC  (u,v) = \C  (u,v) - {\dot C}_1 (u,v) \, \C  (u,1) - {\dot C}_2  (u,v) \, \C  (1,v).
 $$
Moreover, if the above assumption on the existence and continuity of the partial derivatives does not hold, then according to Theorem~4 in \cite{Fermanian-Radulovic-Wegkamp:2004}, the empirical copula process does not converge at all.

In the copula modeling literature, the difference between $\C$ and $\CC$ is often interpreted as ``the price to pay for the fact that the margins are unknown.'' This suggests that if $F$ and $G$ were known, it would be preferable to base the inference on $C_n$ rather than on $\hat{C}_n$. It is shown here, perhaps surprisingly, that the opposite is true under weak positive dependence conditions on $C$. When this happens, procedures based on $\CC$ are thus more efficient than the analogous procedures based on $\C$.

The key result is stated and illustrated in Section~\ref{section:2}, along with a partial extension to the case of negative dependence. In Section~\ref{section:3}, circumstances are delineated under which a dependence parameter, say $\theta = \T(C)$, can be estimated more efficiently by a rank-based estimate ${\hat \theta}_n = \T({\hat C}_n)$ than by the analogous estimate $\theta_n = \T(C_n)$ which exploits the knowledge of the margins. Concluding remarks are given in Section~\ref{section:4}, along with a partial multivariate extension of the main result. All technical arguments are collected in the Appendix.

 \section{Main result}
 \label{section:2}

Following \cite{Esary-Proschan:1972}, suppose that the two continuous random variables $X$ and $Y$ are such that the mappings $t \mapsto \Pr (X \le x \mid Y \le t)$ and $t \mapsto \Pr (Y \le y \mid X \le t)$ are both decreasing in $t$ whatever $x, y \in \mathbb{R}$. These tail monotonicity conditions, jointly referred to as left-tail decreasing\-ness (LTD), imply that the pair $(X,Y)$ satisfies the concept of positive quadrant dependence (PQD). From \cite{Lehmann:1966}, this means that for all $x, y \in \mathbb{R}$,
 $$
 \Pr (X \le x, Y \le y) \ge \Pr (X \le x) \Pr (Y \le y).
 $$

Both PQD and LTD can be stated in terms of the underlying copula only. As shown, e.g., in \cite[Chapter 5]{Nelsen:2006}, PQD holds if and only if $C(u,v) \ge uv$ for all $u, v \in [0,1]$, while the LTD property is verified if for almost all $u, v \in (0,1)$,
 \begin{align}
 \label{eq:1}
&{\dot C}_1 (u,v) \le \frac{C(u,v)}{u} \, ,
&{\dot C}_2 (u,v) \le \frac{C(u,v)}{v} \, .
 \end{align}

Many bivariate models with positive dependence meet Conditions $\eqref{eq:1}$, including the bivariate Normal, Beta, Gamma, Student and Fisher distributions. Other examples are provided by the Cook--Johnson bivariate Pareto, Burr and logistic distributions, Gumbel's bivariate exponential and logistic distributions, the Ali--Mikhail--Haq bivariate logistics, the Clayton, Frank, Plackett and Raftery families of copulas.

As it turns out, the LTD property implies a dominance relation between the asymptotic covariance functions of the empirical processes $\C$ and $\CC$. A formal statement of this fact is given below and proved in the Appendix.

 \begin{proposition}
 \label{p:ltd}
Suppose that $C$ is an LTD copula whose partial derivatives ${\dot C}_1$ and ${\dot C}_2$ exist everywhere and are continuous on $(0,1)^2$. Then for all $u, v, s, t \in [0,1]$,
 \begin{equation}
 \label{eq:2}
\cov \{ \CC   (u,v), \CC   (s,t) \} \le \cov \{ \C  (u,v), \C  (s,t) \}.
 \end{equation}
 \end{proposition}

Inequality $\eqref{eq:2}$ seems to have been intuited in \cite{Charpentier:2007} in the context of copula density estimation; a heuristic explanation was offered, but a formal result was neither stated nor proved. Proposition 4.2 in \cite{Genest-Segers:2009} is also a forerunner of Proposition~$\ref{p:ltd}$ in the case of extreme-value copulas. As shown in \cite{Garralda-Guillem:2000}, the latter are monotone regression dependent in the sense of \cite{Lehmann:1966}; this concept of dependence is stronger than the LTD property.

A simple application of Proposition~$\ref{p:ltd}$ is in the case of independence, where
 $$
\cov \{ \CC   (u,v), \CC   (s,t) \} - \cov \{ \C  (u,v), \C  (s,t) \} = 2uvst  - us(v \wedge t) - vt(u \wedge s)
 $$
is readily seen to be negative for all $u, v, s, t \in [0,1]$. Here is another illustration.

 \begin{example}
 \label{e:fgm}
Consider the Farlie--Gumbel--Morgenstern (FGM) copula with parameter $\theta \in [-1,1]$, which is defined for all $u, v \in [0,1]$ by $C_\theta(u,v) = uv + \theta uv(1-u)(1-v)$. It is easy to check that $C_\theta$ is LTD when $\theta \in [0,1]$; hence Inequality $\eqref{eq:2}$ holds for all $u, v, s, t \in [0,1]$. Analytic expressions for the asymptotic covariances can be derived using \texttt{Maple} but even in this relatively simple case, they are much too long to be displayed here. A graph of $\var \{ \CC (u,v) \} - \var \{ \C (u,v) \}$ is plotted in Figure~1 for the cases $\theta = 1$ (left panel) and $\theta = -1$ (right panel); as the two surfaces look quite similar, their difference ($\text{left}-\text{right}$) is also shown in the middle panel.
 \end{example}

 \begin{figure}
 \begin{center}
 \includegraphics[width=0.30\textwidth]{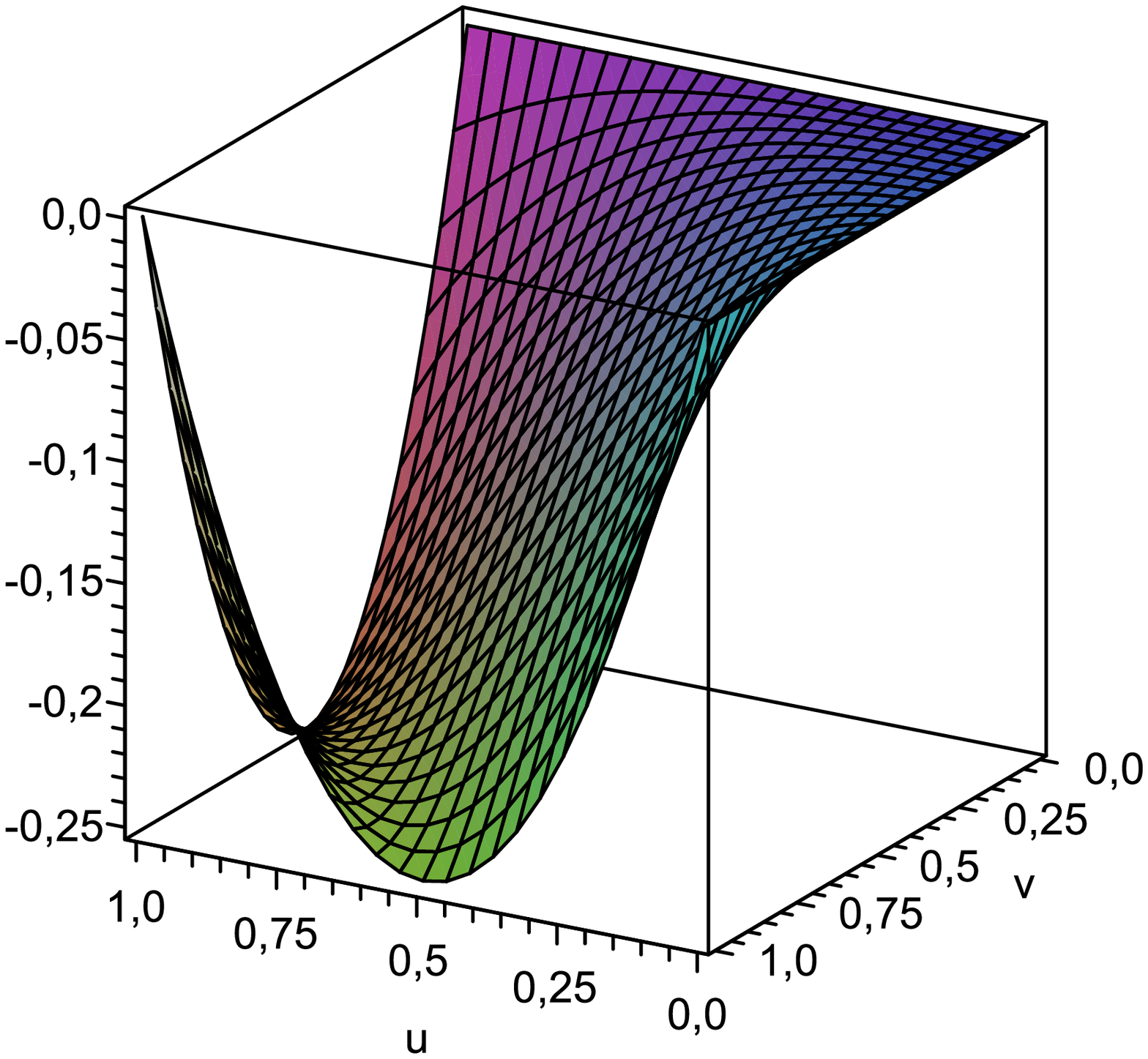}
 \includegraphics[width=0.30\textwidth]{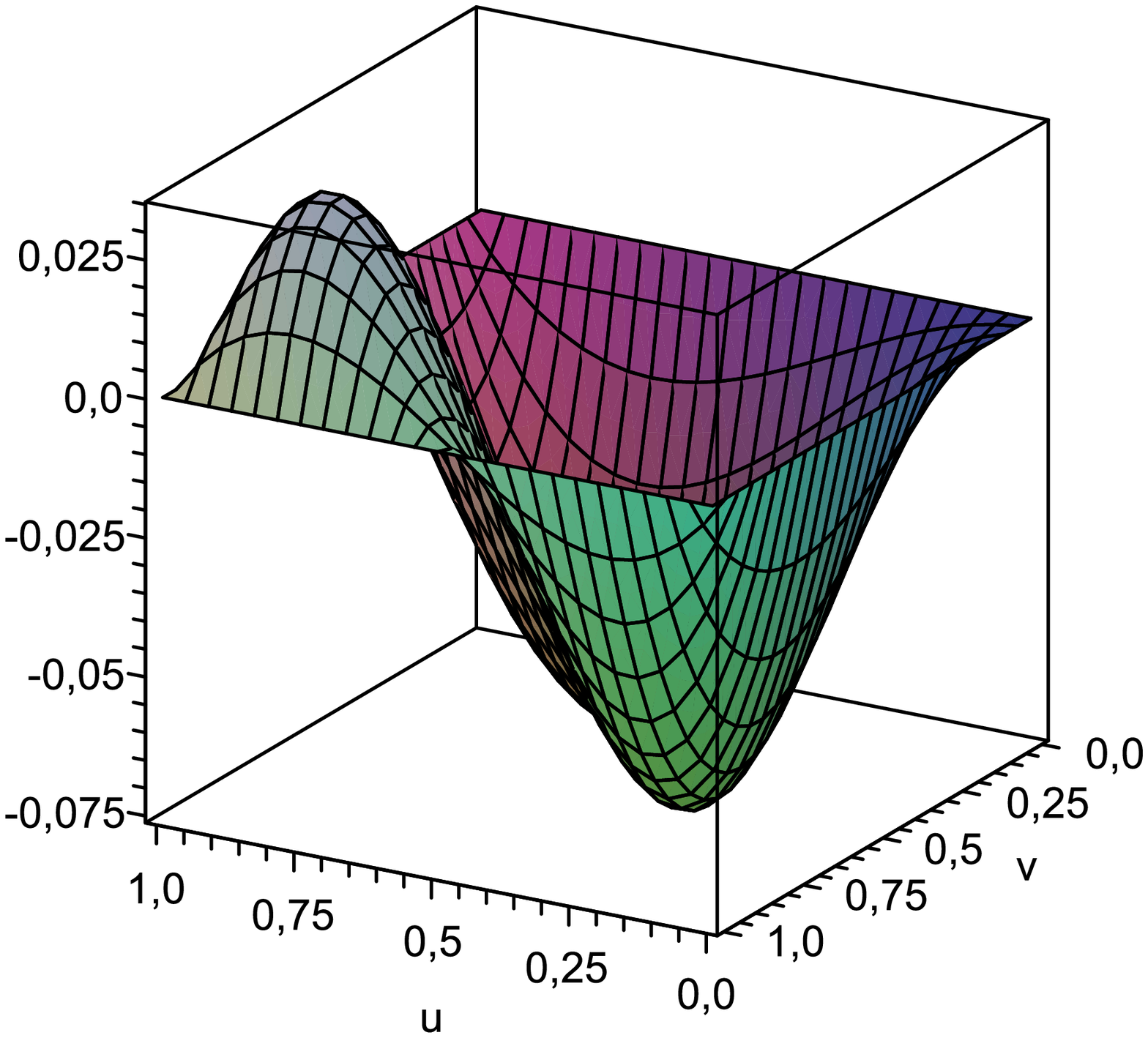}
 \includegraphics[width=0.30\textwidth]{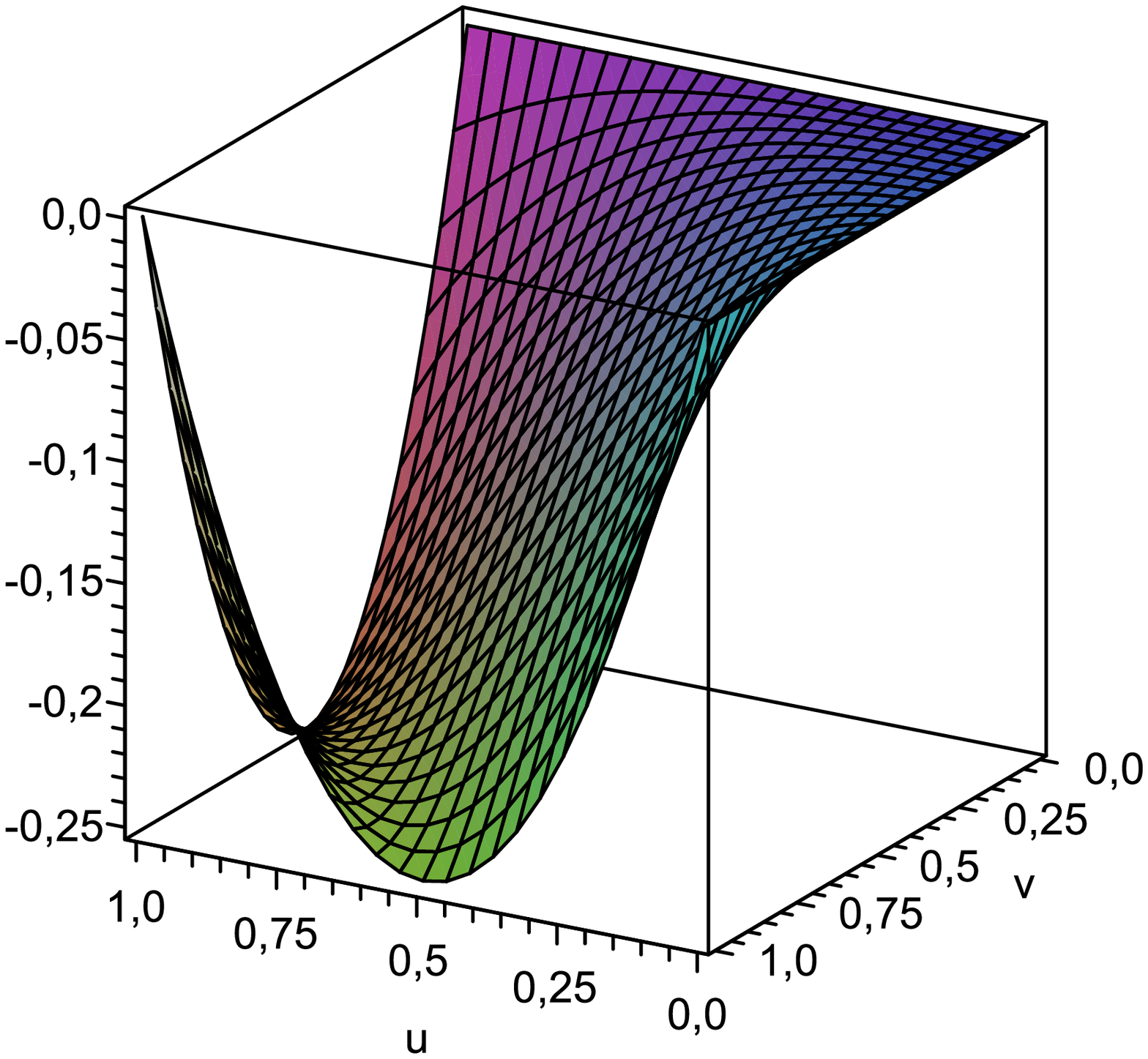}
\caption{\small Graph of $\var \{ \CC (u,v) \} - \var \{ \C (u,v) \}$ for the FGM copula with parameter $\theta = 1$ (left) and $\theta = -1$ (right); the middle graph shows the difference between the two surfaces ($\text{left}-\text{right}$).
 \label{fig:1}}
 \end{center}
 \end{figure}

The right panel of Figure~1 suggests that Proposition~$\ref{p:ltd}$ could possibly be extended to cases where the pair $(X,Y)$ is negative quadrant dependent (NQD), i.e., such that $\Pr (X \le x, Y \le y) \le \Pr (X \le x) \Pr (Y \le y)$ for all $x, y \in \mathbb{R}$. A partial finding along these lines is stated next for copulas that are ``not too negatively dependent,'' in the sense that for all $u, v \in (0,1)$, one has $C(u,v) \le uv$ and
 \begin{align}
 \label{eq:3}
&{\dot C}_1 (u,v) \le 2 \, \frac{C(u,v)}{u} \, ,
&{\dot C}_2 (u,v) \le 2 \, \frac{C(u,v)}{v} \, .
 \end{align}

 \begin{proposition}
 \label{p:nqd}
Suppose that $C$ is an NQD copula whose partial derivatives ${\dot C}_1$ and ${\dot C}_2$ exist everywhere and are continuous on $(0,1)^2$. Further assume that Conditions $\eqref{eq:3}$ hold for all $u, v \in (0,1)$. Then for all $u, v \in [0,1]$,
 \begin{equation}
 \label{eq:4}
\var \{ \CC (u,v) \} \le \var \{ \C (u,v) \}.
 \end{equation}
 \end{proposition}

This result, which is proved in the Appendix, is considerably weaker than Proposition~$\ref{p:ltd}$ because it only yields Inequality $\eqref{eq:2}$ in the case $u=s$ and $v=t$. Condition $\eqref{eq:3}$ does not correspond to any standard notion of negative dependence and it may be insufficient to establish $\eqref{eq:2}$ in full generality. As the following example shows, however, if Condition $\eqref{eq:3}$ does not hold, then the variance inequality $\eqref{eq:4}$ may be violated.

 \begin{example}
 \label{ex:gb}
Consider the Gumbel--Barnett copula with parameter $\theta \in (0,1]$, which is defined for all $u, v \in (0,1)$ by $C_\theta(u,v) = uv \exp(- \theta \log u \log v) \le uv $. Then $C_\theta$ is NQD and for all $u, v \in (0,1)$, one has both ${u{\dot C}_1 (u,v)}/{C(u,v)} = 1 - \theta \log v$ and ${v{\dot C}_2 (u,v)}/{C(u,v)} = 1 - \theta \log u$. It may easily be checked numerically that Condition $\eqref{eq:3}$ is not verified and that Inequality $\eqref{eq:4}$ fails for $u = v$ close to $0$.
 \end{example}

 \section{Consequences for inference}
 \label{section:3}

Proposition~$\ref{p:ltd}$ has intriguing implications for inference about copula-based dependence parameters. To see why, consider the estimation of Blomqvist's medial correlation coefficient using a random sample $(X_1, Y_1), \ldots $, $(X_n, Y_n)$ from a continuous distribution $H$. A multivariate version of this problem was recently studied in \cite{Schmid-Schmidt:2007}.

When $H$ is bivariate and has underlying copula $C$, Blomqvist's beta is given by
 $$
\T_1 (C) = -1 + 4C \left( \frac{1}{2} \, , \frac{1}{2} \right).
 $$

If the margins $F$ and $G$ of $H$ are known, one can then compute $U_i = F(X_i)$ and $V_i = G(Y_i)$ for each $i \in \{ 1, \ldots , n \}$ and the pairs $(U_1, V_1), \ldots , (U_n, V_n)$ form a random sample from $C$. A natural estimator of $\theta = \T_1 (C)$ is then given by
 $$
\theta_n = \T_1 (C_n) = -1 + 4C_n \left( \frac{1}{2} \, , \frac{1}{2} \right) =
-1 + \frac{4}{n} \, \sum_{i=1}^n \1 \left( U_i \le \frac{1}{2} \, , V_i \le \frac{1}{2} \right).
 $$
When the margins are unknown, however, it is still possible to estimate $\theta$ using
 $$
{\hat \theta}_n = \T_1 ({\hat C}_n) = -1 + 4{\hat C}_n \left( \frac{1}{2} \, , \frac{1}{2} \right) =
-1 + \frac{4}{n} \, \sum_{i=1}^n \1 \left( \frac{R_i}{n} \le \frac{1}{2} \, , \frac{S_i}{n} \le \frac{1}{2} \right),
 $$
where for fixed $i \in \{ 1, \ldots , n \}$, $R_i = n {\hat U}_i$ denotes the rank of $X_i$ among $X_1, \ldots , X_n$ and $S_i = n{\hat V}_i$ denotes the rank of $Y_i$ among $Y_1, \ldots , Y_n$.

It follows from the asymptotic behavior of the processes $\C_n$ and $\CC_n$ that the estimators $\theta_n$ and $\hat \theta_n$ are asymptotically unbiased and Gaussian. In other words, there exist centered Normal random variables $\Theta_1$ and $\hat \Theta_1$ such that, as $n \to \infty$, $n^{1/2} (\theta_n - \theta) \rightsquigarrow \Theta_1$ and $n^{1/2} (\hat \theta_n - \theta) \rightsquigarrow \hat \Theta_1$, where $\rightsquigarrow$ denotes weak convergence.

Clearly, $\theta_n$ cannot be used if $F$ and $G$ are unknown. But if they are known, should $\theta_n$ be preferred to ${\hat \theta}_n$? Surprisingly perhaps, Proposition~$\ref{p:ltd}$ implies that when $C$ satisfies Conditions $\eqref{eq:1}$, the rank-based estimator is asymptotically more efficient than its competitor. In other words,
 $$
\var ({\hat \Theta}_1) = 16 \var \left\{ \CC \left( \frac{1}{2} \, , \frac{1}{2} \right) \right\} \; \le \; 16 \var \left\{ \C \left( \frac{1}{2} \, , \frac{1}{2} \right) \right\} = \var (\Theta_1).
 $$

For example if $C$ is the FGM copula with parameter $\theta \ge 0$, one gets
 $$
\var ({\hat \Theta}_1) = \left(1 + \frac{\theta}{4} \right) \left(1 - \frac{\theta}{4} \right) \; \le
\; \left( 1 + \frac{\theta}{4} \right) \left( 3 - \frac{\theta}{4} \right) = \var (\Theta_1).
 $$
The inequality remains valid for $\theta < 0$, as per Proposition~2. In particular, $\var ({\hat \Theta}) = 1$ and $\var (\Theta) = 3$ at independence. The difference is substantial!

Similar conclusions can be drawn for other copula functionals, such as Spearman's foot\-rule, Spearman's rho and Gini's gamma, respectively defined by
 \begin{align*}
\T_2 (C) &= -2 + 6\int_0^1 C(t,t) \, \dd t, \\
\T_3 (C) &= -3 + 12 \int_0^1 \int_0^1 C(u,v) \, \dd u \, \dd v, \\
\T_4 (C) &= -2 + 4 \int_0^1 \{ C(t,t) + C(t,1-t) \} \, \dd t.
 \end{align*}
In each case, the rank-based estimator ${\hat \theta} = \T ({\hat C}_n)$ is more efficient asymptotically than the plug-in estimator $\theta_n = \T (C_n)$, so long as $C$ satisfies Conditions $\eqref{eq:1}$. The efficiency ratio at independence is 5 for $\T_2$ and $\T_4$, and 7 for $\T_3$.

To illustrate the extent of the improvement in cases of dependence, estimators $\T_i (C_n)$ and $\T_i ({\hat C}_n)$ of parameters $\T_i (C)$ were computed for each $i \in \{ 1, \ldots , 4\}$ from 1000 random samples of size $500$ from the bivariate Normal distribution with correlation $\rho = \pm \, 0.5$. Boxplots showing the variation in the estimates are shown in Figure 2. As one can see, the rank-based estimators are preferable in all cases.

 \begin{figure}[h]
 \begin{center}
 \includegraphics[width=0.45\textwidth]{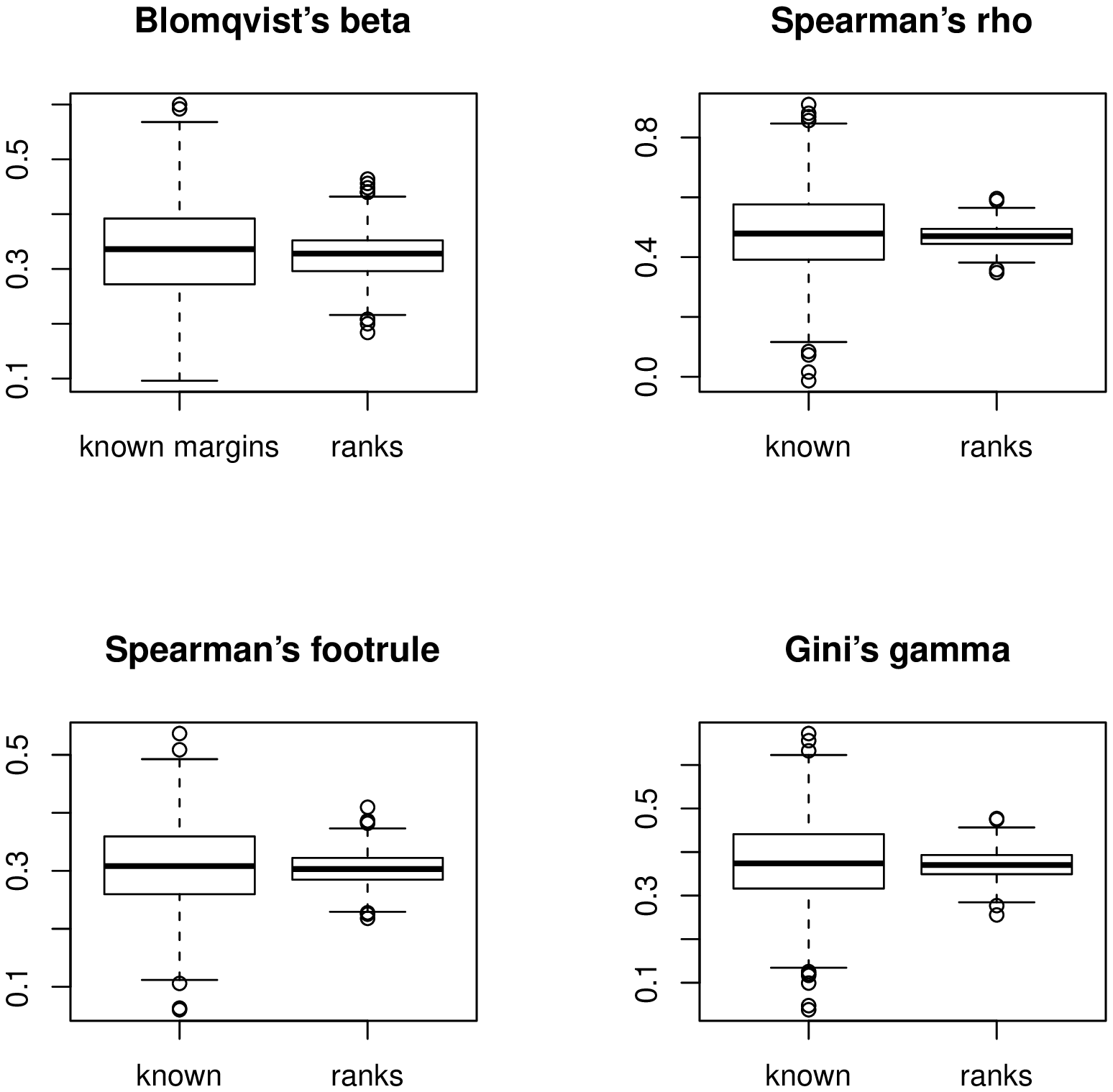}
 \includegraphics[width=0.45\textwidth]{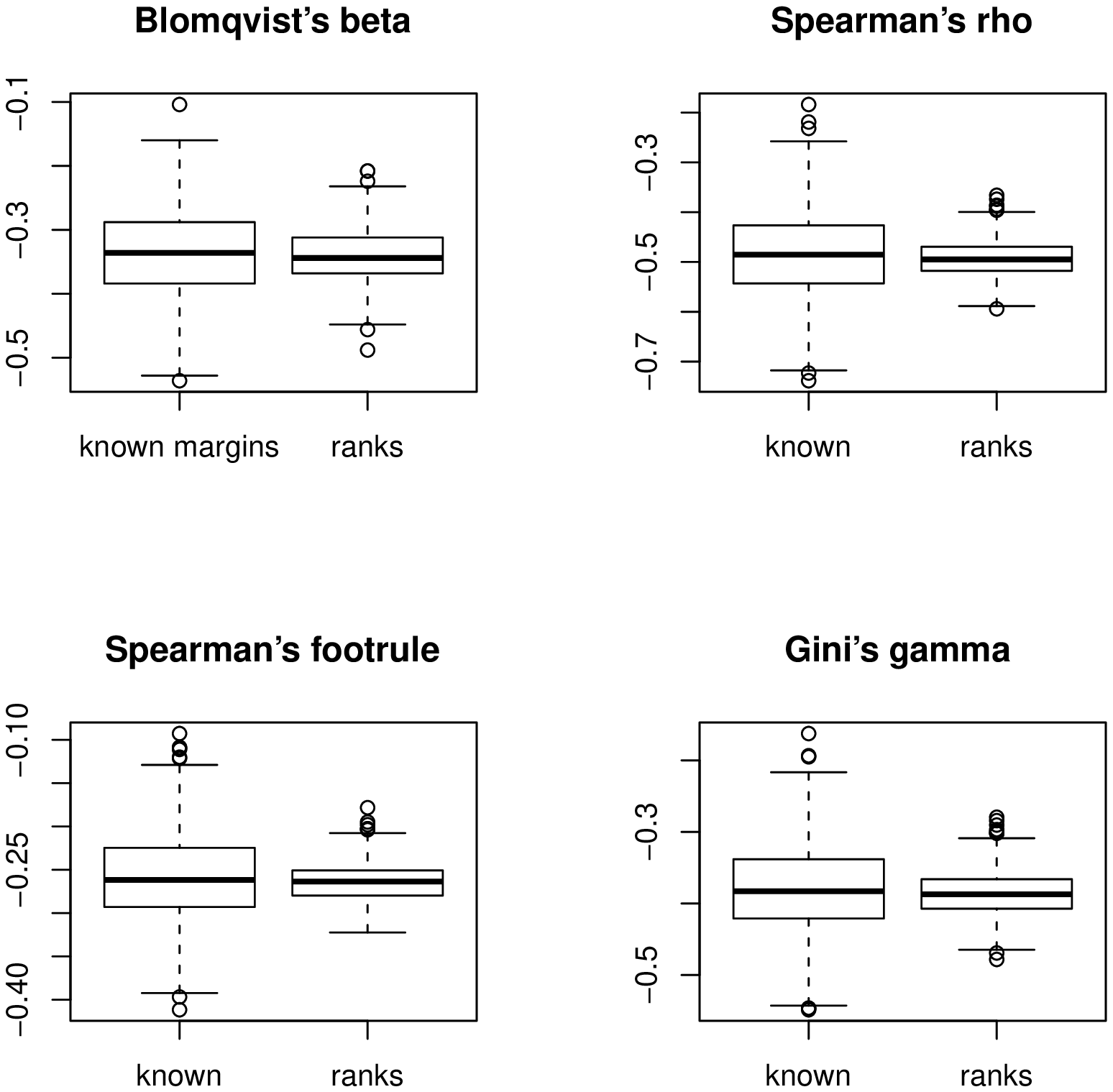}
 \end{center}
\caption{Boxplots for estimates of Blomqvist's beta, Spearman's rho, Spearman's footrule, and Gini's gamma computed from $C_n$ and $\hat{C}_n$, based on 1000 samples of size 500 of the bivariate normal distribution with correlation $\rho = -0.5$ (left) and $\rho = 0.5$ (right).}
 \end{figure}

This observation can be extended as follows by treating $\T$ as a functional on the space $\mathcal{D}$ of c\`adl\`ag functions $\xi: [0,1]^2 \to \mathbb{R}$, equipped with the sup norm.

 \begin{proposition}
 \label{p:deppar}
Suppose that $\T:\mathcal{D} \to \mathbb{R}$ is non-decreasing and Ha\-da\-mard differentiable at any copula $C$, tangentially to the subspace $\TT \subset \D$ of continuous maps. Further assume that as $n \to \infty$, $n^{1/2} \{\T (C_n) - \T(C) \} \rightsquigarrow \Theta$ and $n^{1/2} \{\T ({\hat C}_n) - \T(C) \} \rightsquigarrow {\hat \Theta}$. If $C$ satisfies the conditions in $\eqref{eq:1}$, then $\var ({\hat \Theta}) \le \var (\Theta)$.
 \end{proposition}

The conditions of Proposition~3 are very general and easily verified for many concordance measures \cite{Scarsini:1984}, including functionals $\T_1$ to $\T_4$. However, they do not extend beyond the comparison of plug-in estimators based on non-decreasing functionals $\T$.

To illustrate this point, consider the functional
 $$
\T_{5} (\xi) = 1 + 3 \int_0^1 \{ 2\xi(t,t) - \xi(t,1) - \xi(1,t)\} \, \dd t.
 $$
As $\T_5$ fails to be non-decreasing, one cannot conclude that ${\hat \theta}_n = \T_5 ({\hat C}_n)$ has greater asymptotic efficiency than $\theta_n = \T_5 (C_n)$ as an estimator of $\theta = \T_5(C)$. In other words, if $\Theta_5$ and ${\hat \Theta}_5$ are the weak limits of $n^{1/2} (\theta_n - \theta)$ and $n^{1/2}(\hat \theta_n - \theta)$, respectively, one can then have either $\var ({\hat \Theta}_5) \le \var (\Theta_5)$ or $\var ({\hat \Theta}_5) \ge \var (\Theta_5)$. For instance, if $C$ is the FGM copula, then both inequalities occur for different choices of the parameter $\theta \ge 0$, as can be checked readily using the formulas
 \begin{align*}
& \var ({\hat \Theta}_5) = \frac{2}{5} + \frac{3}{70} \; \theta - \frac{11}{150} \; \theta^2 , && \var (\Theta_5) = \frac{1}{2} - \frac{1}{10} \; \theta - \frac{1}{25} \; \theta^2.
 \end{align*}

A subtlety arises in that although $\T_5$ is not monotone, its restriction to the class of copulas coincides with Spearman's footrule. This is because if $C$ is a copula, $C(t,1) = C(1,t) = t$ for all $t \in [0,1]$. Accordingly, $\T_5 (C_n)$ and $\T_5 ({\hat C}_n)$ are estimators of $\theta = \T_2(C) = \T_5(C)$ which differ from $\T_2 (C_n)$ and $\T_2 ({\hat C}_n)$, respectively. Thus if $\Theta_2$ and ${\hat \Theta}_2$ denote the weak limits of $n^{1/2}\{\T_2 (C_n)-\theta\}$ and $n^{1/2}\{\T_2 ({\hat C}_n)-\theta\}$, respectively, then $\var ({\hat \Theta}_2) \le \var (\Theta_2)$ whereas the same inequality may not hold for $\T_5$. For instance if $C$ is the FGM copula with parameter $\theta \ge 0$, one gets
 $$
\var ({\hat \Theta}_2) = \frac{2}{5} + \frac{3}{70} \; \theta - \frac{11}{150} \; \theta^2 \;\le \; 2 + \frac{2}{5} \; \theta - \frac{1}{25} \; \theta^2 = \var (\Theta_2),
 $$
in accordance to Proposition~$\ref{p:ltd}$.

The fact that $\var ({\hat \Theta}_2) = \var ({\hat \Theta}_5)$ is not a coincidence. It occurs because
 $$
\int_0^1 {\hat C}_n (t,1) \, \dd t = \int_0^1 {\hat C}_n (1,t) \, \dd t = \frac{n-1}{2n} \, ,
 $$
so that the rank-based estimators $\T_5 ({\hat C}_n)$ and $\T_2({\hat C}_n)$ are asymptotically equivalent. To see that $\T_5(C_n)$ and $\T_2(C_n)$ are not asymptotically equivalent, note that
 $$
\T_2 (C_n) = 4 - \frac{6}{n} \sum_{i=1}^n \max (U_i, V_i)
 $$
while
 $$
\T_5 (C_n) = 1 - \frac{6}{n} \sum_{i=1}^n \max (U_i, V_i)
+ \frac{3}{n} \sum_{i=1}^n U_i + \frac{3}{n} \sum_{i=1}^n V_i.
 $$
Therefore, $n^{1/2} \{ \T_5 (C_n) - \T_2 (C_n) \}$ converges weakly to a non-degenerate centered Normal random variable as $n \to \infty$.

More generally if $\T$ and $\T^*$ are smooth functionals that coincide on the class of copulas, then the rank-based plug-in estimators $\T ({\hat C}_n)$ and $\T^* ({\hat C}_n)$ are asymptotically equivalent, while $\T (C_n)$ and $\T^* (C_n)$ are not necessarily so. Indeed if ${C}_n^\maltese$ is the checkerboard copula associated to $\hat{C}_n$ (see, e.g., \cite{Genest-Neslehova:2007}), then for all $u,v \in [0,1]$,
 $$
|\hat{C}_n (u,v) - C_n^\maltese (u,v)| \le \frac{1}{n} \, .
 $$
Furthermore, as $C_n^\maltese$ is a \textit{bona fide} copula, $\T(C_n^\maltese) = \T^* (C_n^\maltese)$ for every integer $n \ge 1$. If the mappings $\T$ and $\T^*$ are differentiable in a neighborhood around $C$, one may conclude that the difference between $\T ({\hat C}_n)$ and $\T^*({\hat C}_n)$ is $O_p(1/n)$.

Finally, note that Kendall's tau is an example of a statistic such that $\T (C_n) = \T ({\hat C}_n)$ for all choices of copula $C$. This is because the concordance or discordance status of pairs $(X_i, Y_i)$ and $(X_j, Y_j)$ is the same, whether it is determined from $(U_i, V_i)$ and $(U_j, V_j)$, or from $({\hat U}_i, {\hat V}_i)$ and $({\hat U}_j, {\hat V}_j)$. Accordingly, the use of ranks does not lead to any efficiency gain or loss in estimating Kendall's tau statistic, whatever $C$.

 \section{Discussion}
 \label{section:4}

Proposition~$\ref{p:ltd}$ provides weak and easy-to-check conditions on a copula $C$ which ensure that the asymptotic covariance of the empirical copula process $\CC_n = n^{1/2} ({\hat C}_n - C)$ is uniformly smaller than the asymptotic covariance of the empirical process $\C_n  = n^{1/2} (C_n - C)$ based on a random sample from $C$. As a consequence, it is shown in Proposition~3 that if a copula-based dependence parameter $\T (C)$ is expressed in terms of a non-decreasing functional of $C$, the plug-in rank estimator $\T ({\hat C}_n)$ has a smaller asymptotic variance than the corresponding estimator $\T (C_n)$ which assumes knowledge of the marginal distributions. A numerical illustration further suggests that the gain in efficiency can be substantial, even in finite samples.

These findings are interesting from a theoretical perspective. They provide broad conditions under which inference on copula-based parameters can be improved when raw observations are replaced by ranks. The gains in efficiency come from the estimation of the marginal distributions, which are nuisance parameters in this context. The present results thus provide a new illustration of the paradoxical effect that nuisance parameters sometimes have on the efficiency of estimators; for additional discussion and examples arising in regression, see \cite{Henmi:2005} and references therein.

At the moment, however, the practical implications of the results remain unclear. As mentioned by a referee, no procedures are currently available for testing that data arise from an LTD copula. The only contribution along these lines seems to be \cite{Denuit-Scaillet:2004}, where a test of the weaker condition PQD is considered. While the development of an LTD test would clearly be of interest, it is beyond the scope of the present work. But it may be worth pointing out that whatever the underlying copula, Inequality $\eqref{eq:2}$ cannot be reversed. For, the limit $\CC$ of the empirical copula process is a ``tucked Brownian sheet'' that vanishes everywhere on the border of $[0,1]^2$, whereas the limit $\C$ of the empirical process is identically zero only on the set $\{ (u,v): u=0$ or $v=0$ or $u = v = 1 \}$. Thus for an arbitrary copula $C$ and for all $u, s \in (0,1)$, one has
 $$
0 = \cov \{ \CC (u,1), \C (s,1)\} < \cov \{ \C (u,1), \C(s,1) \} = u \wedge s - us.
 $$
By continuity, one must also have $\cov \{ \CC (u,v), \C (s,t)\} < \cov \{ \C (u,v), \C(s,t) \}$ when $v$ and $t$ are sufficiently close to $1$.

In future work, it would be of interest to find appropriate conditions under which Inequality $\eqref{eq:2}$ holds for negatively dependent dependence structures. The conclusion from Proposition~2 is considerably weaker. An extension of Proposition~$\ref{p:ltd}$ to arbitrary dimension $d \ge 2$ would also be valuable but seems difficult. For, the proof detailed in the Appendix uses the fact that for all $u, v, s, t \in [0,1]$,
 \begin{equation}
 \label{eq:5}
\cov \{ \CC (u,v), \CC (s,t) \} - \cov \{ \C (u,v), \C (s,t) \} = \sum_{i=1}^4 A_i (u,v,s,t) - \sum_{i=1}^4 B_i (u,v,s,t)
 \end{equation}
for specific choices of functions $A_i$, $B_i$, $i \in \{ 1, \ldots , 4 \}$. Under the conditions of Proposition~$\ref{p:ltd}$, one can show that the right-hand side of $\eqref{eq:5}$ is non-negative by matching each $A_i$ with a specific $B_i$, depending on the relative position of $u, v, s, t \in [0,1]$. In the $d$-variate case, however, there are $d^2$ terms of type $A$ but only $2d$ terms of type $B$, making it impossible to extend the technique used in the proof.

Nevertheless, it is shown in the Appendix that Inequality $\eqref{eq:2}$ continues to hold at independence in higher dimensions. This final result is formally stated below. By a continuity argument, one can thus expect that Proposition~$\ref{p:ltd}$ can be extended to multivariate copulas in an appropriate neighborhood of independence.

 \begin{proposition}
 \label{p:mv}
Let $\C_d$ be the $d$-variate pinned $\pi$-Brownian sheet whose covariance function is given for all $\bu, \bv \in [0,1]^d$ by $\cov \{ \C_d (\bu), \C_d (\bv) \} = \pi (\bu \wedge \bv) - \pi (\bu) \pi (\bv)$,
where for $\bu = (u_1, \ldots , u_d)$, $\bv = (v_1, \ldots , v_d) \in [0, 1]^d$, $\bu \wedge \bv = (u_1 \wedge v_1, \ldots, u_d \wedge v_d)$ and $\pi(\bu) = u_1 \times \cdots \times u_d$. For arbitrary $\bu \in [0,1]^d$, let also
 $$
\CC_d (\bu) = \C_d (\bu) - \sum_{k=1}^d \dot{\pi}_k(\bu) \C_d (\bu_k),
 $$
where $\dot{\pi}_k (\bu) = \partial \pi(\bu)/\partial u_k$ and where $\bu_k$ denotes a vector whose $j$th coordinate is $u_k$ if $j=k$ and $1$ otherwise. Then for all $\bu, \bv \in [0,1]^d$,
 $$
\cov \{ \CC_d (\bu), \CC_d (\bv) \} \le \cov \{ \C_d (\bu), \C_d (\bv) \}.
 $$
 \end{proposition}

 \section*{Appendix: Proofs}

 \begin{proof}[Proof of Proposition~$\ref{p:ltd}$.]
It is obvious from the definition of the limiting process $\CC  $ that Equation $\eqref{eq:5}$ holds with
 \begin{align*}
A_1 (u,v,s,t) &= {\dot C}_1 (u,v) \, {\dot C}_1 (s,t) \cov \{ \C  (u,1), \C  (s,1) \}, \\
A_2 (u,v,s,t) &= {\dot C}_1 (u,v) \, {\dot C}_2 (s,t) \cov \{ \C  (u,1), \C  (1,t) \}, \\
A_3 (u,v,s,t) &= {\dot C}_2 (u,v) \, {\dot C}_1( s,t) \cov \{ \C  (1,v), \C  (s,1) \}, \\
A_4 (u,v,s,t) &= {\dot C}_2 (u,v) \, {\dot C}_2 (s,t) \cov \{ \C  (1,v), \C  (1,t) \}
 \end{align*}
 and
  \begin{align*}
B_1 (u,v,s,t) &= {\dot C}_1 (u,v) \cov \{ \C (u,1), \C (s,t) \}, \\
B_2 (u,v,s,t) &= {\dot C}_2 (u,v) \cov \{ \C (1,v), \C (s,t) \}, \\
B_3 (u,v,s,t) &= {\dot C}_1 (s,t) \cov \{ \C (u,v), \C (s,1) \}, \\
B_4 (u,v,s,t) &= {\dot C}_2 (s,t) \cov \{ \C (u,v), \C (1,t) \}.
 \end{align*}

By symmetry, it can be assumed without loss of generality that $u \le s$. Two cases must be distinguished, according to whether $v \le t$ or $v > t$. If $v \le t$, then
 \begin{align*}
A_1 (u,v,s,t) &= {\dot C}_1 (u,v) \, {\dot C}_1 (s,t) \, u(1-s), \\
A_2 (u,v,s,t) &= {\dot C}_1 (u,v) \, {\dot C}_2 (s,t) \{ C(u,t) - ut \}, \\
A_3 (u,v,s,t) &= {\dot C}_2 (u,v) \, {\dot C}_1 (s,t) \{ C(s,v) - sv \}, \\
A_4 (u,v,s,t) &= {\dot C}_2 (u,v) \, {\dot C}_2 (s,t) \, v (1-t),
 \end{align*}
 and
  \begin{align*}
B_1 (u,v,s,t) &= {\dot C}_1 (u,v) \{ C(u,t) - u \, C(s,t) \}, \\
B_2 (u,v,s,t) &= {\dot C}_2 (u,v) \{ C(s,v) - v \, C(s,t) \}, \\
B_3 (u,v,s,t) &= {\dot C}_1 (s,t) C(u,v)(1-s), \\
B_4 (u,v,s,t) &= {\dot C}_2 (s,t) C(u,v)(1-t).
 \end{align*}

As it happens,
 \begin{align}
 \label{eq:A1}
 \tag{A1}
& A_1 \le B_3, & A_2 \le B_1, && A_3 \le B_2, && A_4 \le B_4,
 \end{align}
where the dependence on $u,v,s,t$ has been suppressed for clarity. Indeed, $A_1 \le B_3$ occurs if and only if ${\dot C}_1  (u,v) {\dot C}_1 (s,t) \, u(1-s) \le {\dot C}_1 (s,t) C(u,v) (1-s)$, which is equivalent to $\eqref{eq:1}$. Similarly, $A_4 \le B_4$.

To get $A_2 \le B_1$, one must check that ${\dot C}_2   (s,t) \{ C(u,t) - ut \} \le C(u,t) -uC(s,t)$.
Using $\eqref{eq:1}$ and the fact that $C(u,t) \ge ut$, one finds
 \begin{align*}
{\dot C}_2   (s,t) \{ C(u,t) - ut \} &\le \frac{C(s,t)}{t} \, \{ C(u,t) - ut \} \\
&= \frac{C(s,t)}{t} \, C(u,t) - u C(s,t) \le  C(u,t) - u C(s,t),
 \end{align*}
where the last inequality is justified because $C(s,t) \le t$. A similar argument yields $A_3 \le B_2$, and hence $\eqref{eq:A1}$ is established.

Now assume that $v > t$. Then
 \begin{align*}
A_1 (u,v,s,t) &= {\dot C}_1 (u,v) \, {\dot C}_1 (s,t) \, u(1-s), \\
A_2 (u,v,s,t) &= {\dot C}_1 (u,v) \, {\dot C}_2 (s,t) \{ C(u,t) - ut \}, \\
A_3 (u,v,s,t) &= {\dot C}_2 (u,v) \, {\dot C}_1 (s,t) \{ C(s,v) - sv \}, \\
A_4 (u,v,s,t) &= {\dot C}_2 (u,v) \, {\dot C}_2 (s,t) \, t (1-v),
 \end{align*}
 and
 \begin{align*}
B_1 (u,v,s,t) &= {\dot C}_1 (u,v) \{ C(u,t) - u \, C(s,t) \}, \\
B_2 (u,v,s,t) &= {\dot C}_2 (u,v) C(s,t) (1-v), \\
B_3 (u,v,s,t) &= {\dot C}_1 (s,t) C(u,v)(1-s), \\
B_4 (u,v,s,t) &= {\dot C}_2 (s,t) \{ C(u,t) - t C(u,v) \}.
 \end{align*}

In that case, it turns out that
 \begin{align}
 \label{eq:A2}
 \tag{A2}
&A_1 \le B_1, & A_2 \le B_4, && A_3 \le B_3, && A_4 \le
B_2.
 \end{align}
Indeed, $A_1 \le B_1$ is equivalent to ${\dot C}_1 (s,t) \, u(1-s) \le C(u,t) - uC(s,t)$. But by $\eqref{eq:1}$,
 $$
{\dot C}_1 (s,t) \, u(1-s) \le \frac{C(s,t)}{s} \, u - u  C(s,t) .
 $$
Thus one can see that
 $$
\frac{C(s,t)}{s} \, u - u C(s,t) \le C(u,t) - u C(s,t)
 $$
whenever $C(s,t)/s \le C(u,t)/u$. The latter holds true because the mapping $u \mapsto C(u,t)/u$ is non-increasing and $u < s$ by the LTD property. Similarly, $A_4 \le B_2$.

Finally, $A_2 \le B_4$ is equivalent to ${\dot C}_1  (u,v) \{ C(u,t) -ut \} \le C(u,t) - t C(u,v)$. Given that $C(u,t) \ge ut$, one can invoke $\eqref{eq:1}$ to write
 $$
{\dot C}_1 (u,v) \{ C(u,t) -ut \} \le \frac{C(u,v)}{u} \, C(u,t) - t C(u,v) \le C(u,t) - t C(u,v),
 $$
where the last inequality is valid because $C(u,v) \le u$. The proof that $A_3 \le B_3$ is similar. The conjunction of $\eqref{eq:A1}$ and $\eqref{eq:A2}$ implies the desired conclusion.
 \end{proof}

 \begin{proof}[Proof of Proposition~$\ref{p:nqd}$.]
Upon setting $s=u$ and $t=v$ in the formulas presented above, one finds
 $$
A_2 = A_3 = {\dot C}_1 (u,v) {\dot C}_2  (u,v)\{ C(u,v) - uv \} \le
0
 $$
for all $u, v \in [0,1]$. Thus it suffices to see that $A_1 \le B_1 + B_3$ and $A_4 \le B_2 + B_4$. This is clearly the case, because $A_1 = (1-u)u {\dot C}_1 ^2 (u,v)$ and $B_1 = B_3 = (1-u)C(u,v){\dot C}_1 (u,v)$, while $A_4 = (1-v)v {\dot C}_2  ^2 (u,v)$ and $B_2 = B_4 = (1-v)C(u,v){\dot C}_2 (u,v)$.
 \end{proof}

 \begin{proof}[Proof  of Proposition~$\ref{p:deppar}$.]
As in \cite[Chapter~20]{vandervaart:1998}, Hadamard differentiability is taken to mean that there exists a continuous linear functional ${\dot \T}_C: \TT \to \mathbb{R}$ such that for every $\xi \in \TT$,
  $$
\lim_{n \to \infty} \frac{\T(C + h_n \xi_n) - \T(C)}{h_n} = \dot{\T}_C(\xi)
 $$
whenever $h_n \downarrow 0$ and $\xi_n \to \xi$ as $n \to \infty$ uniformly. An application of the Functional Delta Method thus implies that, as $n \to \infty$,
 $$
n^{1/2} \{ \T(C_n) - \T(C) \} = n^{1/2} \{ \T (C + n^{-1/2} \C _n ) - \T (C) \} \rightsquigarrow \dot{\T}_C(\C ).
 $$
Consequently, $\Theta = \dot{\T}_C(\C)$ in distribution. Similarly, $\hat \Theta = \dot{\T}_C(\CC)$ in distribution.

Now because the functional $\dot{\T}_C$ belongs to the dual of $\TT$, the Riesz Representation Theorem implies the existence of a bounded Borel measure $\mu_C$ on $[0,1]^2$ such that $\dot{\T}_C(\xi) = \int \xi \, \dd \mu_C$
for all $\xi \in \TT$. Accordingly,
 \begin{align*}
\var (\Theta ) &= \int_{[0,1]^2} \int_{[0,1]^2}
	\cov \{ \C (u, v), \C (s, t) \} \, \dd \mu_C(u, v) \, \dd \mu_C(s, t), \\
\var ({\hat \Theta} ) &= \int_{[0,1]^2} \int_{[0,1]^2}
	\cov \{ \CC (u, v), \CC (s, t) \} \, \dd \mu_C(u, v) \, \dd \mu_C(s, t). \\
 \end{align*}

Finally, the assumption that $\T$ is non-decreasing means that $\xi \le \xi^* \Rightarrow \T(\xi) \le \T(\xi^*)$, where the inequality between functions is understood to hold point\-wise. It then follows that ${\dot \T}_C(\xi) \ge 0$ whenever $\xi \ge 0$. Thus the measure $\mu_C$ must be positive and $\var ({\hat \Theta}) \le \var (\Theta)$ holds as soon as Inequality $\eqref{eq:2}$ holds for all $u, v, s, t \in [0,1]$.
 \end{proof}

 \begin{proof}[Proof of Proposition~$\ref{p:mv}$.]
A simple calculation shows that
 $$
	\left.
		\begin{array}{l}
			C_k (\bv) \cov \{ \C_d (\bu), \C_d (\bv_k) \} \\ [1ex]
			C_k (\bu) \cov \{ \C_d (\bu_k), \C_d (\bv) \} \\ [1ex]
			C_k(\bu) C_\ell(\bv) \cov \{ \C_d (\bu_k), \C_d (\bv_k) \}
		\end{array}
	\right\}
	= \pi(\bu)\pi(\bv)\, \frac{u_k \wedge v_k - u_k v_k}{u_kv_k} \, .
 $$
Given that $\cov \{ \C_d (\bu_k), \C_d (\bv_\ell) \} = 0$ whenever $k \neq \ell$, one gets
 $$
\cov \{ \CC_d (\bu), \CC_d (\bv) \} - \cov \{ \C_d (\bu), \C_d (\bv) \} =
- \pi(\bu) \pi(\bv) \sum_{k=1}^d \frac{u_k \wedge v_k - u_k v_k}{u_kv_k} \, ,
 $$
which is clearly negative for all $\bu$ and $\bv \in [0,1]^d$.
 \end{proof}

 \section*{Acknowledgements}

The authors are grateful to Ivan Kojadinovic and Johanna Ne\v{s}lehov\'a for useful discussion. Funding in support of the first author's work was provided by the Natural Sciences and Engineering Research Council of Canada, by the Fonds qu\'e\-b\'e\-cois de la re\-cher\-che sur la na\-tu\-re et les tech\-no\-lo\-gies, and by the Ins\-ti\-tut de fi\-nan\-ce ma\-th\'e\-ma\-ti\-que de Montr\'eal. Funding in support of the second author's work was provided by Research supported by IAP research network grant P6/03 of the Belgian government (Belgian Science Policy) and by ``Pro\-jet d'ac\-tions de re\-cher\-che con\-cer\-t\'ees'' number 07/12/002 of the Com\-mu\-nau\-t\'e fran\-\c{c}ai\-se de Bel\-gi\-que, granted by the Aca\-d\'e\-mie uni\-ver\-si\-taire de Lou\-vain.

{}
 \end{document}